\pdfoutput=1
\documentclass[12pt]{article}
\usepackage{setspace}
\usepackage[a4paper]{geometry}
\usepackage{amssymb, amsmath, amsthm}
\usepackage{mathrsfs}
\usepackage{graphicx}
\usepackage[pdftex,usenames,dvipsnames]{color}
\usepackage{subfigure}

\newcommand{\norm}[1]{\left|\left| #1 \right|\right|}

\newcommand{\ltwo}[1]{\left|\left| #1 \right|\right|_{L^2}}
\newcommand{\linfty}[1]{\left|\left| #1 \right|\right|_{L^\infty}}
\newcommand{\abs}[1]{\left| #1 \right|}

\newcommand{\dx}{\mathrm{d}x}
\newcommand{\dy}{\mathrm{d}y}
\newcommand{\ds}{\mathrm{d}s}
\newcommand{\dd}{\mathrm{d}}

\newcommand{\dg}{\mathrm{d}g}

\newcommand{\R}{\mathbb{R}}
\newcommand{\N}{\mathbb{N}}

\renewcommand{\P}{\mathrm{P}}

\newcommand{\Fel}{\mathcal{F}_\mathrm{elastic}}

\newcommand{\F}{\mathcal{F}}
\renewcommand{\div}{\operatorname{div}}

\newcommand{\avgint}{-\!\!\!\!\!\!\int}

\newcommand{\Laplace}{\Delta}
\newcommand{\T}{\mathrm{T}}
\newcommand{\db}[1]{\left[\!\left[ #1 \right]\!\right]}

\newcommand{\eps}{\varepsilon}

\theoremstyle{plain}
\newtheorem{proposition}{Proposition}[section]

\newtheorem{theorem}[proposition]{Theorem}

\theoremstyle{definition}

\newtheorem{hypothesis}[proposition]{Hypothesis}

\theoremstyle{remark}
\newtheorem{remark}[proposition]{Remark}

\title{Effective behavior of an interface propagating through a periodic elastic medium}
\author{Patrick W. Dondl\\ {\it \small Durham University, Durham, UK}\\
{\it \small Email: patrick.dondl@durham.ac.uk}\\
\\ and\\ \\ 
Kaushik Bhattacharya\\
{\it \small California Institute of Technology, Pasadena, USA}\\
{\it \small Email: bhatta@caltech.edu}}

\begin{document}
\maketitle

\begin{abstract}
We consider a moving interface that is coupled to an elliptic equation in
a heterogeneous medium.  The problem is motivated by the study of displacive solid-solid phase
transformations.   We show that a nearly flat interface is given by the graph of the function
$g$ which evolves according to the equation $g_t (x) = -(-\Laplace)^{1/2}g (x) + \varphi(x, g(x)) + F$.
This equation also arises in the study of dislocations and fracture.  We show in the periodic setting
that such
interfaces exhibit a stick-slip behavior associated with pinning and depinning.  Further, we present some numerical 
evidence that the effective velocity of the  phase boundary scales as the square-root of the excess 
macroscopic force above the depinning transition.
\end{abstract}

\label{sec:shallow}
\section{Introduction}
\begin{figure}
\centering
\resizebox{0.6\textwidth}{!}{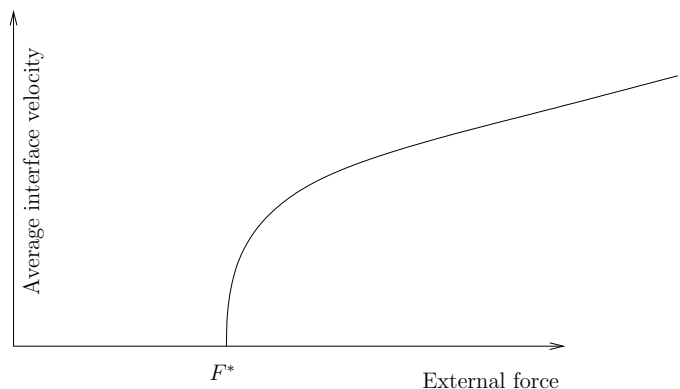}
\caption[Pinning-depinning behavior, as calculated from a one-dimensional model in~\cite{Bhattacharya_99b}]{Pinning-depinning behavior as calculated from a one-dimensional model in~\cite{Bhattacharya_99b}. The interface is stuck up to a critical force $F^*$, and breaks free above it with a particular scaling.  For large $F$, the average velocity is linear in the applied force.}
\label{fig:depinning_1d}
\end{figure}

Hysteresis is ubiquitous in materials science, and is associated with nucleation and propagation of interfaces and defects.  This paper concerns the propagation of interfaces immersed in an elastic medium.  In particular, the paper is motivated by phase boundaries in solids that undergo a displacive phase transformation such as the martensitic phase transformation.  In these transformations, one has phase boundaries across which the crystal structure changes without any diffusion or loss of compatibility.  Many interesting properties of such materials, like the shape-memory effect, are associated with the nucleation and evolution of these phase boundaries.  As the interface propagates, the change in crystal structure potentially gives rise to elastic fields.  Thus, one has a moving interface problem that is coupled to an elasticity problem.  Further, every material contains defects like non-transforming precipitates which makes the medium inhomogeneous.  An important question is the role of these defects, and this motives our current work.  Similar issues arise in ferroelectricity, ferromagnetism, and other phenomena.

There is a well-developed framework to study the evolution of martensitic phase boundaries, and this is described in detail in the recent monograph of Abeyaratne and Knowles \cite{Abeyaratne_06a}.  Briefly, one defines a thermodynamic driving force either through the rate of dissipation or through the variation in the total energy with respect to the position of the interface, and then postulates a kinetic relation that relates the driving force to the normal velocity of the interface.  Microscopic theories suggest that the kinetic relation has viscous character passing smoothly through the origin \cite{Abeyaratne_91b,Purohit_03a}.  Such a kinetic relation predicts that the hysteresis goes to zero as the rate of loading goes to zero.  However, experiments clearly show otherwise: the hysteresis does not go to zero with loading rate and instead settles on a non-zero value independent of loading-rate for slow enough rates.  Such observations suggest a stick-slip behavior where the interface is stationary below a critical driving force and moves freely above it.  It is often suggested that pinning of the phase boundary by defects is responsible for this transition from microscopic viscous to macroscopic stick-slip behavior.



 A one-dimensional calculation, as found in~\cite{Abeyaratne_96a, Bhattacharya_99b} illustrates how a local wiggly potential can pin a phase boundary and lead from a linear kinetic relation to a stick-slip behavior. Assume a bar with a 1-periodic local driving force $\varphi(x)$ (smooth and with non-degenerate global maximum and minimum), and assume that the velocity of the interface is given as $v=\varphi+F$, where $F$ is the constant external applied force. The amount of time it takes for the interface to travel one period can now easily be calculated to be
\begin{equation}
T = \int_0^1 \frac{\dg}{F+\varphi(g)},
\end{equation}
if $F> -\min \varphi$ or  $F< -\max \varphi$. Otherwise (i.e., if $-\min \varphi < F < -\max \varphi)$, the time is infinite and the interface is stuck.   Further, close to the critical $F$, say $ F \approx -\max \varphi$, the interface is slow only in a few isolated points but propagating freely everywhere.  This implies, under some non-degeneracy and regularity conditions, that the effective velocity scales as the square-root of the excess force.  Thus, the effective velocity $\bar{v} = \frac{1}{T}$ of the interface now exhibits a behavior of the form shown in Figure~\ref{fig:depinning_1d}.  A rigorous proof of the transition from a viscous microscopic kinetic law to a rate independent evolution through the interaction with a wiggly potential can be found in~\cite{Mielke:2011ji}. 
The question whether such a stick-slip behavior is also observed in models for phase transformations in higher dimension motivates this work.

We present a sharp interface model for the quasistatic evolution of a martensitic phase boundary in higher dimensions in Section 2.  We limit ourselves to the scalar anti-plane shear setting (where the displacements are scalars) though the ideas and results hold for the general case.  In this model, a free boundary separates two material phases. Each phase is characterized by a distinct transformation or stress-free strain where the elastic energy density admits its minimum.    We also assume that the material contains a number of non-transforming precipitates.  Importantly, both the phases as well as the non-transforming precipitates have the same energy.  A similar model was studied by Craciun~\cite{Craciun_01a}.

We then derive an approximate model for a nearly flat interface.  We show, using methods of $\Gamma$-convergence, that the elastic energy of a nearly flat interface is approximated by the 
$H^{1/2}$-norm of a function whose graph describes the interface (Theorem 2.1).  We also argue that at low volume fraction, the precipitates give rise to a local forcing which scales similarly to the elastic energy.  We thus conclude that the interface is described by the graph of a function $g$ which is governed by the equation
\begin{equation} \label{model}
g_t = -(-\Laplace)^{1/2}g + \varphi(x_1, ..., x_n, g(x_1, ..., x_n)) + F
\end{equation}
for a given $\varphi:{\mathbb R}^{n+1} \rightarrow {\mathbb R}$ with zero mean.
On the periodic domain we consider, this equation may be compactly written by its Fourier series,
\begin{equation}
\hat{g}_t(k) = - \abs{k} \hat{g}(k) + \hat{\varphi}(k) + \hat{F}(k).
\end{equation}
From now on, $\hat{g}$ indicates the Fourier series of the periodic function $g$. Note that the equation is still nonlinear, since the driving force $\varphi$ depends on $g$. 

While we derive this model from phase transformations, it has also been used to study dislocations~\cite{Koslowski_02a} (see~\cite{Forcadel:2009vn} for some rigorous analysis of that model) as well as fracture~\cite{Ponson_09a}.  In fact, a very similar model is derived in~\cite{Forcadel:2009ec} as the homogenized limit of an interacting system of individual dislocations. Front-type solutions as limits of reaction-fractional-diffusion equations are derived in~\cite{Imbert:2009ut}. A closely related 
parabolic model, 
\begin{equation} \label{parabolic}
g_t = \Laplace g + \varphi + F
\end{equation}
has been used to study pinning of surface energy dominant interfaces by defects.  The large physics literature has concentrated on the situation where $\varphi$ is random, and has shown using scaling arguments and numerical simulation that these equations lead to a pinning/depinning transition with a critical exponent which varies from situation to situation~\cite{Barabasi_95a}.  Dirr and Yip~\cite{Dirr_06a} presented a rigorous analysis of the parabolic model (\ref{parabolic}) in the periodic setting ($\varphi$ is periodic).   A rigorous analysis of the random case remains the topic of ongoing research (see for
example~\cite{Dondl_09f, Dondl_09a}).

We study the behavior of the solutions of (\ref{model}) in Section 3 where $\varphi$ is ($1-$)periodic following Dirr and Yip~\cite{Dirr_06a}.  We show that there is a critical $F^\ast \ge 0$ such that (\ref{model}) admits a stationary solution for all $F \le F^\ast$ (Theorem 3.5).  Further, for each $F>F^\ast$, there exists an unique $T$ such that (\ref{model}) admits a space-time periodic solution (Theorem 3.7 and Proposition 3.9).  Thus, we may regard $1/T$ as the effective velocity of the interface.   In Section 4, we discuss the behavior of the effective velocity near the depinning transition and present some numerical examples indicating that the effective velocity scales as the square-root of the excess force $F-F^*$ in the case of a smooth, non-degenerate heterogeneity.



\section{A model of phase transformations}

\subsection{Phase transformations in the presence of defects}

\begin{figure}
\begin{center}
\resizebox{0.4\textwidth}{!}{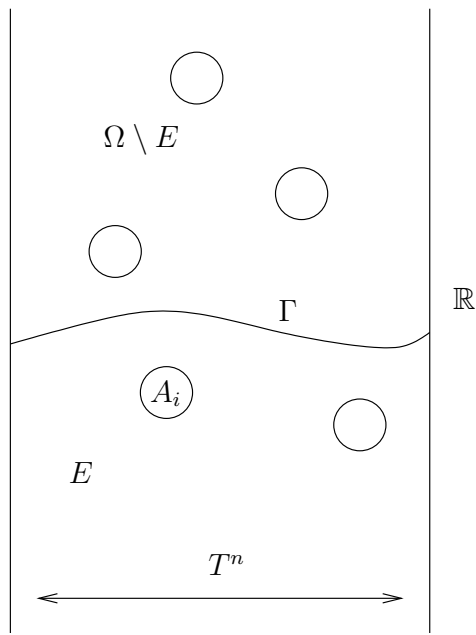}
\end{center}
\caption{A phase boundary in a strip containing non-transforming precipitates.}
\label{fig:strip}
\end{figure}

We consider a model proposed by Craciun \cite{Craciun_01a}.  Since we are interested in the overall propagation, we consider the domain to be a strip, $\Omega = T^{n} \times \R$ where $T^{n}$ is an $n$-dimensional torus as shown in Figure~\ref{fig:strip}.  The domain $\Omega$ is divided into two parts, $E$ and $\Omega\setminus E$ occupied by two phases, and separated by the phase boundary $\Gamma$ of codimension 1.   The domain also contains a number of non-transforming precipitates, occupying the set $\bigcup_i A_i$.  The two phases are characterized by two stress-free strains $\xi^\pm \in {\mathbb R}^{n+1}$, and the non-transforming precipitates are characterized by the stress-free strain $\xi^0$.  We further assume that all phases and precipitates have equal elastic modulus (which we take to be identity without loss of generality).  Thus the elastic energy of domain subjected to the displacement $u: \Omega \rightarrow {\mathbb R}$ is given by
\begin{equation}
\Fel = \int_\Omega \frac{1}{2}\norm{ \nabla u - \xi_E }^2 \dx \dy .
\end{equation}
where
\begin{equation}
\xi_E(x,y) = \left\{ \begin{array}{ll} 
\xi^+ & (x,y) \in E \setminus \bigcup_i A_i \\
\xi^- & (x,y) \in (\Omega \setminus E) \setminus \bigcup_i A_i \\
\xi^0 & (x,y) \in \bigcup_i A_i.
\end{array} \right.
\end{equation}
and $\nabla$ denotes gradient with respect to $(x,y)$.

For a given interface $\Gamma$, we obtain the displacement by minimizing the energy subject to appropriate boundary conditions.  This displacement satisfies the Euler-Lagrange equation
\begin{equation} \label{EL}
\Delta u = \div \xi_E .
\end{equation}
We then say that the interface $\Gamma$ evolves in a specified manner that reduces the (optimal) elastic energy \cite{Abeyaratne_06a}.

In this section, we seek to find an approximation for this particular model.  To motivate the approximation, consider the situation where one does not have any non-transforming precipitates and the interface that minimizes the total energy. It is easy to show in the current scalar setting that the interface is a plane with normal $n = (\xi^+ - \xi^-)/|\xi^+ - \xi^-|$ \footnote{In the vectorial setting of linearized elasticity, there are two possible normals \cite{Kohn_91a}}.  Therefore, we expect that an arbitrary interface will soon become almost planar with this normal and then evolve in an almost planar manner. Further, in the presence of defects, we expect the interface to be distorted close to them due to the elastic fields created by the defects but be largely planar away from them. Thus, if the concentration of defects is small, the interface remains largely planar.  This is consistent with numerical observations~\cite{Dondl_07a}. All of this motivates us to seek an approximation for the model in the case that the interface is almost planar.

Due to the linearity of the Euler-Lagrange equation (\ref{EL}), one can split the transformation strain and the displacement into
components depending only on the interface and on the precipitates, respectively. Take
\begin{equation}
u =  u^\Gamma + u^\P; \quad \xi_E = \xi^\Gamma +  \xi^\P
\end{equation}
and fix
\begin{equation}
\xi^\Gamma(x) = \left\{ \begin{array}{ll} 
\xi^+ & x \in E\\
\xi^- & x \in (\Omega \setminus E),
\end{array} \right.
\end{equation}
and
\begin{equation}
\xi^P(x) = \left\{ \begin{array}{ll} 
\xi^0-\xi^+ & x \in \bigcup_i A_i \cap E\\
\xi^0-\xi^- & x \in \bigcup_i A_i  \cap (\Omega \setminus E).
\end{array} \right.
\end{equation}
The displacements $u^\P, u^\Gamma$ solve the Euler-Lagrange equations associated with $\xi^\Gamma, \xi^\P$,
\begin{equation} \label{EL2}
\Laplace u^\Gamma = \div \xi^\Gamma; \quad \Laplace u^\P = \div \xi^\P.
\end{equation}
This fixes the functions $u^\Gamma$ and $u^\P$ up to an affine component.

Substituting this decomposition back into the elastic energy, we can expand it as follows:
\begin{eqnarray}   \label{eq:energysplit} 
\int_\Omega \frac{1}{2}\norm{\nabla u - \xi_E}^2 &=& 
\int_\Omega \frac{1}{2}\norm{\nabla u^\Gamma - \xi^\Gamma}^2 \\
&& + \int_\Omega (\nabla u^\P - \xi^\P) \cdot \nabla u^\Gamma  \notag \\
&& - \int_\Omega (\nabla u^\P - \xi^\P) \cdot \xi^\Gamma. \notag\\
&& + \int_\Omega  \frac{1}{2}\norm{\nabla u^\P - \xi^\P}^2 \notag 
\end{eqnarray}
The first term in \eqref{eq:energysplit} is the energy associated with the interface $\Gamma$ in the absence of any precipitates.  We shall call this the {\it self-energy} of the interface, and we find an approximation for it in section~\ref{sec:Gamma_convergence}. We can use the Euler-Lagrange equation (\ref{EL2}) to show that the second term is identically zero, assuming that the affine component of $u^\Gamma$ is chosen appropriately. 
The remaining terms describe the interaction of the precipitates with the interface, which is dealt with in a heuristic manner in section~\ref{sec:approx_interaction}.

We seek approximations to the self energy and the interaction energy when the interface is almost planar.  We specialize to the case when 
\begin{equation}
\label{eq:xipm}
\xi^\pm = (0, \dots, 0, \pm 1/2)^\T
\end{equation}
and the preferred normal is $n = (0, \dots, 0, 1)^T$.  We assume that the phase boundary is the graph of a $H^1$ and essentially bounded function $g$, i.e.,
\begin{equation}
\Gamma = \left\{ (x_1, x_2, ..., x_n, y) \textrm{ such that } y = g(x_1, ..., x_n) \right\},
\end{equation}
and the transformed domain $E = \{ (x,y) : y< g(x) \}$. 


%


\subsection{Approximation of the self-energy}
\label{sec:Gamma_convergence}


We first show that the self-energy, the first term on the right hand side of (\ref{eq:energysplit}), scales as $\varepsilon^2$ if the interface is rescaled by a factor $\varepsilon$ and in the limit may be approximated by the one half the square of the $H^{1/2}$ semi-norm of $g$.  Since this term is independent of the precipitates, we ignore these for now.
It is convenient to subtract the piecewise constant, curl-free, function $\frac{1}{2}(1-H(y))$\footnote{$H(y)$ denotes the Heavyside step function.} from the transformation strain $\xi_E$,
in order to make it supported in a bounded region around $y=0$. This does not change the energy of minimizers of the energy functional, since one can simply
subtract the respective integral from $u$. From now on, we thus assume that, for a given function $\tilde{g}\colon T^n \to \R$,
\begin{equation}
\xi_{\tilde{g}}^n(x,y) = \left\{ \begin{array}{ll} 0  & \textrm{for $y\le0$ and $y\le \tilde{g}$} \\
-1 & \textrm{for $y\le 0$ and $y>\tilde{g}$} \\
1 & \textrm{for $y>0$ and $y\le \tilde{g}$} \\
0 & \textrm{for $y> 0$ and $y>\tilde{g}$} \end{array}\right.
\end{equation}
and $\xi_{\tilde{g}}(x,y) = \left(\begin{array}{c} 0 \\ \vdots \\ 0 \\  \xi_{\tilde{g}}^n(x,y) \end{array}\right)$.
In the following, we take $\tilde{g}= \varepsilon g$ for $g \in H^1(T^n) \cap L^\infty(T^n)$, $\varepsilon$ small, and prove the convergence of the funtional
\begin{equation}
\F_\varepsilon( u ) = \int_\Omega \frac{1}{2}\norm{\nabla u - \frac{1}{\varepsilon} \xi_{\varepsilon g} }^2,
\end{equation}
in the sense of $\Gamma$-convergence with respect to the usual $L^1_\mathrm{loc}$ metric, to the limit functional
\begin{equation}
\F(u) = \left\{ \begin{array}{ll} \int_{\Omega \setminus \{y=0\}} \frac{1}{2}\norm{\nabla u}^2 & \textrm{if $\db{u}=g$ a.e.} \\
	\infty & \textrm{otherwise}. \end{array} \right.
\end{equation}
Both functionals are assumed to be infinite if $u\notin H^1_\mathrm{loc}$.
The expression $\db{u}$ denotes the jump of $u$ across $\{y=0\}$ (i.e., the difference of the respective traces). We write
\begin{eqnarray*}
u_\varepsilon^+ (x,y) &=& u_\varepsilon(x, y+\varepsilon) \\
u_\varepsilon^- (x,y) &=& u_\varepsilon(x, -y-\varepsilon)  
\end{eqnarray*}
Without loss of generality, we assume here that $\norm{g}_{L^\infty} \le 1$, otherwise the definition of $u_\varepsilon^\pm (x,y)$ needs to be adapted so 
that the interface does not penetrate outside the cutout region.

\begin{theorem}[Approximation of the energy]
\label{thm:Gamma}
Let  $(\varepsilon_j)_{j\in \N}$ 
be a decreasing sequence of positive real numbers converging to zero. The following assertions hold.\\
i) Consider functions $u_{\varepsilon_j}$ such that $\F_{\varepsilon_j}( u_{\varepsilon_j} )$ is uniformly bounded. Then there exists a subsequence (which we relabel and 
also index by $j$) and functions $u^{\pm}\in H^1_\mathrm{loc}(T^n \times (0,\infty))$ , such that
\begin{eqnarray*}
u_{\varepsilon_j}^+ &\rightharpoonup& u^+ \\
u_{\varepsilon_j}^- & \rightharpoonup& u^-
\end{eqnarray*}
weakly in $H^1_\mathrm{loc}(T^n \times (0,\infty))$ (and thus strongly in $L^1_\mathrm{loc}$), modulo a constant function. \\
ii) Given a sequence $u_{\varepsilon_j}$, such that $u_{\varepsilon_j}^\pm \to u^\pm$
in $L^1_\mathrm{loc}(T^n \times (0,\infty))$, we have
\begin{equation}
\liminf \F_{\varepsilon_j}( u_{\varepsilon_j} ) \ge \F( u ),
\end{equation}
where $u(x,y) = u^\pm(x,\pm y)$ for $y>0$ $(+)$ or $y<0$ $(-)$, respectively. \\
iii) Given $u\in H^1(\Omega \setminus \{y=0\})$, there exists a sequence $u_{\varepsilon_j}$ so that $u_{\varepsilon_j}^ + \rightarrow u$ in
$L^1_\mathrm{loc}(T^n\times(0,\infty))$,  $u_{\varepsilon_j}^- \rightarrow u(x,-y)$ in $L^1_\mathrm{loc}(T^n\times(0,\infty))$, and
\begin{equation}
\limsup \F_{\varepsilon_j}( u_{\varepsilon_j} ) \le \F( u ).
\end{equation}
\end{theorem}
\begin{proof}
\textit{i)} Clearly, $\int_{T^n\times(0,\infty)} \norm{\nabla u_\epsilon^\pm}^2$ is uniformly bounded. This yields the desired compactness.
\\
\textit{ii)} If $\db{u} = g$, the result is immediate from the definition of the energies. 
What remains to show, is that the elastic energy necessarily
blows up if the difference of the respective traces of $u_\varepsilon(x, \varepsilon)$
and $u_\varepsilon(x, -\eps)$ does not converge to the function $g$. Note that, for all $\varepsilon$, one can calculate
\begin{eqnarray}
\F_\varepsilon(u_\varepsilon) &\ge& 
\int_{T^n} \int_{-\eps}^\eps \frac{1}{2}\norm{\nabla u_\varepsilon - \frac{1}{\varepsilon} \xi_{\varepsilon g}}^2\notag \\
&\ge&  \int_{T^n} \int_{-\eps}^\eps  \frac{1}{2}\abs{ \frac{\partial}{\partial y}u_\varepsilon - \frac{1}{\varepsilon} H(\varepsilon g(x)-y) + \frac{1}{\eps}H(-y)}^2 \notag\\
&\ge& \frac{1}{2\varepsilon} \left(  \int_{T^n} \int_{-\eps}^\eps  \frac{1}{2}\abs{\frac{\partial}{\partial y}u_\varepsilon - \frac{1}{\varepsilon} H(\varepsilon g(x)-y) + \frac{1}{\eps}H(-y) } \right)^2 \label{eq:gamma_liminf_jensen} \\
&\ge& \frac{1}{2\varepsilon} \left(\int_{T^n} \frac{1}{2}\abs{  \int_{-\eps}^\eps  \frac{\partial}{\partial y}u_\varepsilon - \frac{1}{\varepsilon} H(\varepsilon g(x)-y) + \frac{1}{\eps}H(-y)} \right)^2 \notag \\
&=& \frac{1}{2\varepsilon} \left(  \int_{T^n} \frac{1}{2}\abs{ (u_\varepsilon(\cdot,\varepsilon)-u_\varepsilon(\cdot, -\eps)) - g } \right)^2. \label{eq:gamma_liminf_trace}
\end{eqnarray}
Jensen's inequality was used in~\eqref{eq:gamma_liminf_jensen}. In~\eqref{eq:gamma_liminf_trace}, $u_\varepsilon(\cdot, y)$ denotes the trace of $u_\varepsilon$ on $(\cdot, y)$. The equality
holds, because $u_\varepsilon$ is in $H^1$, and thus admits a representative that is absolutely continuous on a.e.~line~\cite{Evans_92b}.
This energy cannot be bounded, unless 
$u_\varepsilon(\cdot,\varepsilon^\alpha)-u_\varepsilon(\cdot, -\eps^\alpha)$ converges to $g$ in $L^1$,
and thus the jump of $u$ across $\{y=0\}$ equals $g$, since weak convergence in $H^1$ implies weak convergence of the trace. \\
\textit{iii)} We first assume that the trace of $u$ at $y=0$ from below, denoted by $T^-u$, is in $H^1$.
Take 
\begin{equation}
u_\varepsilon(x,y) =  \left\{ \begin{array}{ll} u(x,y+\varepsilon) & \textrm{for $y\le -\varepsilon$,} \\
T^-u(x)+\frac{y}{\varepsilon}H(\varepsilon g(x)-y) & \textrm{for $-\varepsilon<y<\varepsilon$,} \\
u(x,y-\varepsilon) & \textrm{for $y \ge \varepsilon$}. \end{array}\right.
\end{equation}

Note that this function is in $H^1_\mathrm{loc}(\Omega)$, since the traces at $\pm \varepsilon$ match, inside the strip the $y$-derivative of $u_\varepsilon$ is bounded by $\frac{1}{\varepsilon}$, and the $x$-derivatives of $u_\varepsilon$ are bounded by the sum of the $x$-gradient of $T^-u$ and the
$x$-gradient of $g$, both of which were assumed to be in $L^2$.

The elastic energy $\F_\varepsilon(u_\varepsilon)$
outside the strip of thickness $\varepsilon$ remains exactly equal to $\F(u)$. The  $y$-derivative
of the function $u_\varepsilon$ equals the $n$-component of $\xi_{\varepsilon g}$, and the $x$-derivative remains bounded by the (absolute) sum of that of $T^-u$ and that of $g$. Thus the integral over the vanishing domain $T^n \times (-\varepsilon, \varepsilon)$ goes to zero.

In order to obtain the result for arbitrary $u \in H^1_\mathrm{loc}$, one can employ the usual density argument. Functions with $H^1$-trace are dense and one can approximate with the energy bounded by $F(u)$.
\end{proof}

The above theorem shows that for a nearly flat interface, the energy due to the shape of the phase boundary itself is equal to the energy of
a function with a jump of the appropriate height $g$. It is well known~\cite{Garroni_05a}, that the minimum attained at $\tilde{u}$ of this energy  is equal to one half the $H^{1/2}$ seminorm squared of $g$, or, 
\begin{equation}
\label{eq:approx_self_energy}
\min_{\stackrel{u \in H^1(\R^\pm \otimes T^n)}{\db{u}_{y=0}=g}} \F(u) = \frac{1}{2}[g]^2_{H^{1/2}}.
\end{equation}

\begin{remark}
Note that we have not proved the $\Gamma$-convergence result for the case that $g$ is only in $H^{1/2}$, since under this weaker assumption a recovery sequence can not be found. However, as we will see in Section~\ref{sec:pinning}, solutions of the evolution problem considered will have the required regularity.
\end{remark}

\subsection{Approximation of the interaction energy}
\label{sec:approx_interaction}

We now turn to the interaction energy, the last two terms of the right hand side of (\ref{eq:energysplit}), 
\begin{equation}
{\mathcal F}^\mathrm{int} :=  \int_\Omega -(\nabla u^\P - \xi^\P) \cdot \xi^\Gamma +   \frac{1}{2}\norm{\nabla u^\P - \xi^\P}^2.
\end{equation}

Scaling again the transformation strains $\xi^\P$ and $\xi^\Gamma$ by a factor of $\frac{1}{\eps}$ and assuming an interface height and precipitate radii rescaled  by $\eps$, one can see that an assumption
of a precipitate density of scale $\eps$ yields an order one term for the interaction energy as well.


Inherently, like the self-energy, ${\mathcal F}^\mathrm{int}$ is of course non-local. However,  the variation of the interaction energy term, as long as the interface does not intersect any of the precipitates, is
local. Assuming again that the interface is the graph of a function $g$, and $\xi^\pm$ given as in~\eqref{eq:xipm} one can write it as
\begin{equation}
\delta_g {\mathcal F}^\mathrm{int} =  [\nabla u^P(x,g(x)) - \xi^P(x,g(x))] \cdot (0,\dots,0,1)^\T =: f(x,g(x)),
\end{equation}
see~\cite{Abeyaratne_06a}.
From now on we assume that the interaction force for small scarcely scattered precipitates can be written in the above local form, ignoring the non-locality when passing through inclusions.
\subsection{The approximate model}

The previous sections show that the energy of an interface described by the graph of a function $g$ may be approximated as
\begin{equation}
{\mathcal E} =  \frac{1}{2}[g]^2_{H^{1/2}} + \int_{T^n}\int_0^{g(x)} f(x,y) \,\dx\dy.
\end{equation}

The evolution of the interface is described as an $L^2$ gradient flow of $g$ with respect to the energy. Of course, $\Gamma$-convergence of the energy does not necessarily imply convergence of solutions to the gradient flow, however, we take the $\Gamma$-convergence result in Theorem~\ref{thm:Gamma} as a good indicator that the evolution of the interface can be approximated by a gradient flow of the limit energy. Since the variation of the $H^{1/2}$ seminorm yields the square root of the Laplacian, we obtain
\begin{equation}
g_t(x,t) = -(-\Laplace)^{1/2}g(x,t) + f(x,g(x,t)).
\end{equation}
This equation may be compactly written by its Fourier series,
\begin{equation}
\label{eq:approx_model}
\hat{g}_t(k,t) = - \abs{k} \hat{g}(k,t) + \widehat{f(\cdot,g(\cdot,t))}(k).
\end{equation}
Note that the equation is still nonlinear, since the driving force $f$ depends on $g$.

Finally, we assume that the distribution of precipitates is periodic in the direction of propagation of the interface. Thus, the forcing
\begin{equation}
f(x,g(x)) =  [\nabla u^P(x,g(x)) - \xi^P(x,g(x))] \cdot (0,\dots,0,1)^\T
\end{equation}
can be split up into a periodic term $\varphi(x,g(x))$ with zero average and a constant term $F$ stemming from boundary conditions at $\pm\infty$, so that
\begin{equation}
f(x,g(x)) = \varphi(x,g(x)) +F.
\end{equation}

\section{Stick-slip behavior}
\label{sec:pinning}

In this section, we study the stationary equation
\begin{equation}
\label{eq:stationary_F}
0 = -(-\Laplace)^{1/2} g + \varphi + F
\end{equation}
and the evolution equation
\begin{equation}
\label{eq:evolution_F}
g_t = -(-\Laplace)^{1/2} g + \varphi + F,
\end{equation}
together with an initial condition, and denote them by~\eqref{eq:stationary_F}$_{F}$ and~\eqref{eq:evolution_F}$_{F}$, respectively,
indicating the dependence on the behavior on the external force $F\ge 0$.  Our strategy and results closely mirror those of Dirr and Yip~\cite{Dirr_06a} who considered the Laplacian case. In this section we will make the following assumption on the wiggly force
$\varphi \colon T^n \times \R \to \R$.
\begin{hypothesis}
The interaction force $\varphi(x,y)$ is periodic in $y$, i.e., $\varphi(x, y) = \varphi(x, y+1)$,
has vanishing average, i.e., $\int_{T^n \times [0, 1)} \varphi(x, y) = 0$ and is Lipschitz continuous in both $x$ and $y$.
\label{hyp:phi1}
\end{hypothesis}

\subsection{Solutions of the evolution equation}

In this section we collect some properties of solutions of the evolution problem~\eqref{eq:evolution_F}$_F$.

\begin{proposition}
\label{thm:existence}
For any initial condition $g(\cdot,0) = g_0 \in C^0$, the quasistatic evolution problem~\eqref{eq:evolution_F}$_F$ admits a unique global in time classical solution. Classical solutions to~\eqref{eq:evolution_F}$_F$ admit a comparison principle. Furthermore, for $g_0 \in L^2$, a mild solution to the evolution problem exists and is classical for any $t>0$.
\end{proposition}
\begin{proof}
We first note that the evolution equation~\eqref{eq:evolution_F}$_F$ can be cast in the form
of~\cite{Droniou:2006jt}, equation (2), by periodically extending $\varphi$ to $\R^n$. Since in our model the function $\varphi$ is assumed to be Lipschitz, it follows from~\cite{Droniou:2006jt}, Theorem 5, that the problem admits a unique Lipschitz continuous viscosity solution, and thus admits a comparison principle. We further remark that the fractional Laplacian generates an analytic semigroup on the space of continuous functions on $T^n$ (for a reference see for example~\cite{Yosida:1996us}, Chapter IX.11, and note that this property is well known for the `regular' Laplacian), and thus the problem also admits a classical solution that can be found via a variation of constants-formula. This also shows that the viscosity solution for the problem is periodic.

For an initial condition only in $L^2$, note that the fractional Laplacian also generates a semigroup on $L^p$, $p\ge 2$, with solutions in $H^{1,p}$, the domain of $-(-\Laplace^{1/2})$ for values in $L^p$. We thus can recursively find spaces of higher and higher integrability for our mild solution at positive time until we obtain a Sobolev-embedding into the space of continuous functions.
\end{proof}

The next proposition concerns an energy estimate of the solution to the evolution equation.\begin{proposition}[Energy estimate]
\label{thm:energy}
Fix $\tau >0$. Then there exists a constant $C$ depending only on $f$ and $\tau$, such that for any solution of~\eqref{eq:evolution_F}$_F$ we have
\begin{equation}
[g(t)]_{H^{1/2}} \le C e^{-t} \norm{g_0}_{L^2} + C \quad\textrm{for}\quad t > \tau.
\end{equation}
\end{proposition}
\begin{proof}
First note that the mild solution of~\eqref{eq:evolution_F} is given by the variation of constants formula
\begin{equation}
\hat{g}(k,t) = e^{ -k t } \hat{g}_0(k) + \int_0^t e^{-k (t-s)} \widehat{f(\cdot,g(\cdot,s))}(k) \,\ds
\end{equation}
for $k \in \{0,1,2,\dots\}$.

We have, for $t>1$, by Plancherel's theorem,
\begin{eqnarray}
[g(t,k)]_{H^{1/2}}^2 &=& \sum_{k=1}^\infty k \abs{\hat{g}(t,k)}^2 \\
&=&  \sum_{k=1}^\infty k \abs{e^{ -k t } \hat{g}_0(k) + \int_0^t e^{-k (t-s)} \widehat{f(\cdot,g(\cdot,s))}(k) \,\ds}^2 \\ 
&\le&  \sum_{k=1}^\infty k \abs{e^{ -k t } \hat{g}_0(k)}^2 + \sum_{k=1}^\infty k \abs{\int_0^{t-1} e^{-k (t-s)}  \widehat{f(\cdot,g(\cdot,s))}(k) \,\ds}^2 \\
&& +  \sum_{k=1}^\infty k \abs{\int_{t-1}^t e^{-k (t-s)}  \widehat{f(\cdot,g(\cdot,s))}(k) \,\ds}^2 \\
&\le& C_1 e^{-t} \norm{g_0}_{L^2}+ C \\
&& + \sum_{k=1}^\infty k  \abs{\sqrt{\int_{t-1}^t e^{-2k (t-s)}\,\ds} \sqrt{\int_{t-1}^t \abs{ \widehat{f(\cdot,g(\cdot,s))}(k)}^2 \,\ds}}^2 \label{eq:evol_reg_hoel}\\
&\le& C_1 e^{-t} \norm{g_0}_{L^2}+ C \\
&& + C' \sum_{k=1}^\infty \int_{t-1}^t \abs{ \widehat{f(\cdot,g(\cdot,s))}(k)}^2 \,\ds \\
&\le& C_1 e^{-t} \norm{g_0}_{L^2} + C_2.
\end{eqnarray}
If $t\le 1$, the integral from $0$ to $t-1$ can be disregarded and the integral from $t-1$ to $t$ runs from $0$ to $t$, with no change in the
estimates. In this calculation, we have used H\"older's inequality in~\eqref{eq:evol_reg_hoel} and the fact that $\int_{t-1}^t e^{-k(t-s)} \,\ds = \frac{1-e^{-k}}{k}$ thereafter.
\end{proof}

\subsection{Existence of pinned and space-time periodic solutions}
\label{sec:pinned-and-periodic}

First we assert the existence of a stationary solution for zero external driving force.
\begin{proposition}
\label{thm:stationary_zero}
Under Hypothesis~\ref{hyp:phi1}, equation~\eqref{eq:stationary_F}$_0$ admits a weak solution in $H^{1/2}$.
\end{proposition}
\begin{proof}
Note first tha $H^{1/2}$ is compactly embedded in $L^2$, independent of dimension. One can thus find a minimizer of the energy 
\begin{equation}
E(g) =  \frac{1}{2}[g]_{H^{1/2}}^2 - \int_{T^n} \int_0^{g(x)} \varphi(x, \gamma) \,\dd \gamma \dx
 \end{equation}
 among functions in $H_c^{1/2} = \{u \in H^{1/2} : \int_{T^n} u = c \}$ with average $c$.
Now denote by
\begin{equation}
G(c) := \min_{g\in H_c^{1/2}} E(g)
\end{equation}
the energy depending on the fixed average $c$. Since $f$ has zero average, we have $G(c+1)=G(c)$. Furthermore, $G$ is Lipschitz by a simple comparison argument. Therefore, $G$ admits a minimum for some $c_0 \in \R$. The  function $g_{c_0}$ is a weak solution to~\ref{eq:stationary_F}$_0$, since it minimizes the corresponding energy.
\end{proof}
\begin{proposition}
\label{prop:regularity_of_stationary}   
Any weak solution ~\eqref{eq:stationary_F}$_F$, $F\ge 0$ is classical (and thus also a
stationary viscosity solution).
\end{proposition}
\begin{proof}
Plugging the weak solution $g$ of the stationary equation into the variation of constants formula for the mild solution of the evolution problem, we find that $g$ is a stationary mild solution, since
$$
e^{ -k t } \hat{g}(k) + \int_0^t e^{-k (t-s)} \widehat{f(\cdot,g(\cdot))}(k) \,\ds =  e^{ -k t } \hat{g}(k) + \int_0^t e^{-k (t-s)} k \,\hat{g}(k) \ds =  \hat{g}(k)
$$
for $k \in \{0,1,2,\dots\}$. From the regularity properties of mild solutions with initial conditions in $L^2$, we find that the stationary solution must be classical.
\end{proof}

The following theorem asserts the existence of a threshold force, up to which---but not above which---a stationary solution exists.
\begin{theorem}[Existence of a threshold force]
There exists $F^*\ge 0$ such that equation~\eqref{eq:stationary_F}$_F$
admits a solution for all $F \le F^*$, while it has no solution for $F>F^*$.
\end{theorem}
\begin{proof}
Consider 
$$
\Phi = \{ F \ge 0 \textrm{ such that~\eqref{eq:stationary_F}$_F$ has a solution} \}.
$$
Clearly, because of Proposition~\ref{thm:stationary_zero}, $\Phi \neq \emptyset$. Also, if $F > \sup \varphi $, then~\eqref{eq:stationary_F}$ _F$ has no solution. Define, therefore, $F^* = \sup\{\Phi\} < \infty$. Two things remain to be shown in order to establish the result: \\
i) $F^* \in \Phi$. \\
ii) There is a solution to~\eqref{eq:stationary_F}$ _F$ for all $F<F^*$. \\
Proof of i): Consider a sequence $F_j \nearrow F^*$ and corresponding solutions $g_j$ of~\eqref{eq:stationary_F}$_{F_j}$ such that $0 \le \int_{T^n} g_j \le 1$. Such a sequence exists, otherwise $F^*$ could not be the supremum of $\Phi$. Since $\{F_n\}$ is bounded, Proposition~\ref{prop:regularity_of_stationary} asserts compactness of $\{g_n\}$ in $H^1$.  Therefore, there is a converging subsequence whose limit again satisfies~\eqref{eq:stationary_F}$_{F^*}$.\\
Proof of ii): Consider $0 < F < F^*$, and the solutions $g^*$ and $g_0$ to the stationary equation with external driving force $F^*$ and $0$, respectively. By the periodicity of $\varphi$, and the continuity of the solutions we can assume $g^*>g_0$. These solutions are also super- and subsolutions to~\eqref{eq:stationary_F}$_F$, respectively. Therefore, there exists a solution to~\eqref{eq:stationary_F}$_F$.
\end{proof}

The next theorem proves the existence of a space-time periodic solution to~\eqref{eq:evolution_F}$_F$ above the critical force.
\begin{theorem}[Existence of a space-time periodic solution]
For each $F>F^*$, there exists a unique $0 <T(F) < \infty$ and a unique function $g(x,t)$ satisfying~\eqref{eq:evolution_F}$_F$ such that
\begin{equation}
\label{eq:periodic}
g(x,t+T) = g(x,t)+1.
\end{equation} 
\end{theorem}
\begin{proof}
Proposition~\ref{thm:existence} shows global existence of a classical solution. We will split the solution up into the evolution of the average and the evolution of the deviation of the average and use Schauder's fix point theorem to prove the existence of a solution satisfying~\eqref{eq:periodic}.

Define $p(t) = \avgint g(x,t) \,\dx$ and $\xi(x,t) = g(x,t) - p(t)$. The functions $p$ and $\xi$ then satisfy
\begin{eqnarray}
\dot{p}(t) &=& \avgint f_g(x,t) \,\dx, \\
\hat{\xi}_t(k,t) &=&  -k \hat{g}(k,t) +  \widehat{f(\cdot,g(\cdot,t))}(k) \quad \textrm{for $k\ge1$}.
\end{eqnarray}
Consider the initial condition $\xi(x,0) = \xi_0(x) \in L^2$, and $p(0) = 0$. Since $F>F^*$, no stationary solution exists. Therefore, there is $T(\xi_0) > 0$, such that $p(T(\xi_0)) = p(0)+1$, and a constant $\tau$, independent of $\xi_0$, such that $T(\xi_0) > \tau$. This follows from the fact that $\abs{\dot{p}} \le \linfty{f}$. We also have $T(\xi_0) < \infty$ for all $\xi_0 \in L^2$, since otherwise one could find a stationary solution to the problem at $F>F^*$, namely the pointwise limit for $t\to \infty$ of the solution to the evolution problem with initial condition $\xi_0$.

Now consider the nonlinear operator that advances the solution $\xi$ in time, such that
\begin{equation}
\mathcal{T}(\xi_0) = \xi(\cdot, T(\xi_0)).
\end{equation}
and we have
\begin{equation}
\hat{\xi}(T,k) = e^{-kT}\xi_0(k) + \int_0^T e^{-k(t-s)} \widehat{f(\cdot,g(\cdot,s))}(k) \,\ds \quad k\ge1. 
\end{equation}
From a similar calculation as in the proof of Proposition~\ref{thm:energy} it is clear that $\ltwo{\mathcal{T}(\xi_0)} \le e^{-\tau}\ltwo{\xi_0} + C$, therefore, for $A = C/(1-e^{-\tau})$, this operator maps the set $\ltwo{ \xi_0 } < A $ onto itself. The regularity estimate in Proposition~\ref{thm:energy} also shows that $\xi(T)$ is bounded in $H^{1/2}$ independent of $\xi_0$, as long as $\ltwo{ \xi_0 } < A$. The operator $\mathcal{T}$ is thus compact and an application of Schauder's fix point theorem yields the existence of a time-space periodic solution which is classical by Proposition~\ref{thm:existence}. Uniqueness of $T$ and $g$ follow from the Proposition~\ref{prop:unique} below.
\end{proof}

\begin{proposition}[Uniqueness of the space-time periodic solution]
\label{prop:unique}
The time-period $T$ for a solution to~\eqref{eq:evolution_F}$_F$ satisfying
\begin{equation}
g(x,t+T) = g(x,t)+1
\end{equation}
is unique. Also, the solution itself is unique up to a time-shift so that, given two solutions $g_1$ and $g_2$ there exists $t_0$ such that $g_1(x,t) = g_2(x,t+t_0)$.
\end{proposition}
\begin{proof}
Assume that there exist two space-time periodic solutions $g_1$ and $g_2$, with time constants $0 < T_2 < T_1 < \infty$. Since the solutions are continuous and invariant under translations by an integer, one can find $N \in \N$ such that, for some time $t_0$, one has $g_1(\cdot, t_0) \le g_2(\cdot, t_0) + N$. But since $T_2 < T_1$ there exists a time $T$ after which the two solutions would have passed each other, contradicting the comparison principle.

Now, consider a solution $G_1$ with the initial condition $g_1(\cdot, t_0)$ and a solution $G_2$ with initial condition $g_2(\cdot, t_0)$. There have to exist a time $T$, an integer $N$, and a point $x_0$ such that we have
\begin{equation}
G_1(x_0, T) = G_2(x_0, t_0) + N \quad \textrm{and} \quad G_1(\cdot, T) \le G_2(\cdot, t_0) + N,
\end{equation}
i.e., the solutions have to touch at some time. Evolving both solutions in time from there on, one can see that they touch again after one time period. This, however, again contradicts the comparison principle, unless $G_1(\cdot, T+t) = G_2(\cdot, t)$.
\end{proof}

\section{Power laws near the depinning transition}
\label{sec:numerical_depinning}

In this section we investigate the behavior of the average interface velocity near the depinning transition. We first note that in the ODE model studied in~\cite{Abeyaratne_96a, Bhattacharya_99b}, namely
\begin{equation}
\label{eq:ODE_dep}
\dot{g}(t) = \varphi(g(t)) + F,
\end{equation}
for a Lipschitz continuous, 1-periodic function $\varphi \colon \R \to \R$ and $F>F_c = -\min \varphi$, one can obtain any behavior between 
$$
\bar{v} = \frac{C}{-\log{\abs{F-F_c}}} \quad \textrm{and} \quad \bar{v} = C\abs{F-F_c}^1,
$$
in leading order for some constant $C>0$, where $\bar{v} = \left(\int_0^1 \frac{\mathrm{d}g}{F + \varphi(g)} \right)^{-1}$ is the average interface velocity. The limiting case of a logarithmic behavior can be produced by taking
$$
\varphi(g) = 2\abs{g-1/2}-1,
$$
which yields $\bar{v} = \frac{1}{\log(F) - \log(F-1)}$ and thus behaves like $ \frac{1}{-\log{\abs{F-F_c}}}$ for $F$ close to $F_c = 1$.

On the other hand, consider 
$$
\varphi(g) = \left\{ \begin{array}{ll} -4g &\textrm{for $0\le g<1/4$} \\
-1 & \textrm{for $1/4\le g<3/4$} \\
-1+4(g-3/4) &  \textrm{for $3/4\le g<1$},
\end{array}\right.
$$
which yields $\bar{v} = \frac{2(F-1)}{1-\log(F)+F\log(F)+\log(F-1)-F\log(F-1)}$. One can easily see that $\bar{v} = 2(F-1)$ in leading order for $F>1$.

For the non-degenerate case of $\varphi$ admitting a non-vanishing second derivative at its minimum, the square root power-law of $\bar{v} = C\abs{F-F_c}^{1/2}$ has been shown in~\cite{Abeyaratne_96a, Bhattacharya_99b} for the ODE case~\eqref{eq:ODE_dep}. As stated in the introduction, under similar non-degeneracy conditions, this power-law behavior has been proved for the case of a parabolic model~\eqref{parabolic} by Dirr and Yip in~\cite{Dirr_06a}.

\subsection{Numerical method for the investigation of the depinning transition}
\begin{table}
\centering
\begin{tabular}{cc}
Number of Fourier coefficients & 1024\\
Length over which the interface velocity is averaged &  4\\
Initial upper bound for $F^*$ in bisection & 0.5\\
Initial lower bound for $F^*$ in bisection & 0 \\  
Threshold for accuracy of  $F^*$ & $2\cdot 10^{-9} $\\
Coefficient of elastic force & 0.1 \\
Time step &  $1\cdot 10^{-3}$ \\
Threshold for stuck interface & $1\cdot 10^{-14}$
\end{tabular}
\caption[Parameters used for the numerical examination of the depinning transition]{Parameters used for the numerical examination of the depinning transition}
\label{tab:general_depinning}
\end{table}

In this section we show some numerical results yielding a good agreement with the square-root power law behavior in the case of smooth, non-degenerate obstacles. As it was shown in the beginning of this section, this behavior is not generic -- in fact, also for our model any depinning behavior seen in the ODE model~\eqref{eq:ODE_dep} can be reproduced by simply picking a heterogeneity $\varphi(x,t)$ independent of $x$.

It follows that the depinning behavior depends very sensitively on the discretization of the pinning force, since a piecewise linear discretization of a smooth obstacle field for example can easily destroy the original power-law behavior. So, starting with an initial configuration $g(0) = 0$ and a fixed applied load $F$, we numerically integrate equation~\eqref{eq:approx_model} using an explicit first-order Euler scheme. This scheme is the most appropriate here, since a very small time step has to be chosen in order not to `jump over' critically pinned states near the depinning transition, which immediately renders implicit schemes useless. We also avoid using a higher-order scheme, since  we want to sample the pinning force at short intervals for the numerical integration and not approximate it by a higher-order polynomial.

The elastic force in our scheme, however, is calculated to high accuracy using discrete Fourier transforms. In order to obtain a smooth discretization of the pinnign fore, the heterogeneity is constructed using a cubic B-spline. Once the interface has traveled a certain length on average (and never got stuck on the way), the final time is recorded.  This way, a relation between the average velocity $\bar{v}$ and $F$ is obtained. The interface is considered stuck if the $L^2$ norm of the driving force $f$ drops below a certain threshold. This `inner loop' is repeated with $F$ chosen each time through a bisection algorithm, thus giving new upper and lower bounds for the critical $F^*$ at each run. The program terminates after a certain accuracy for determining $F^*$ has been reached. In Table~\ref{tab:general_depinning} the standard parameters for the simulation can be found.

\subsection{Simulations}
\begin{figure}
\centering
\subfigure[Force distribution used to examine the general depinning behavior. The `up then down'  bumps model the $x_2$ derivative of an attractive potential well.]
{
\label{fig:standard_pinningforce}
\includegraphics[width=0.43\textwidth]{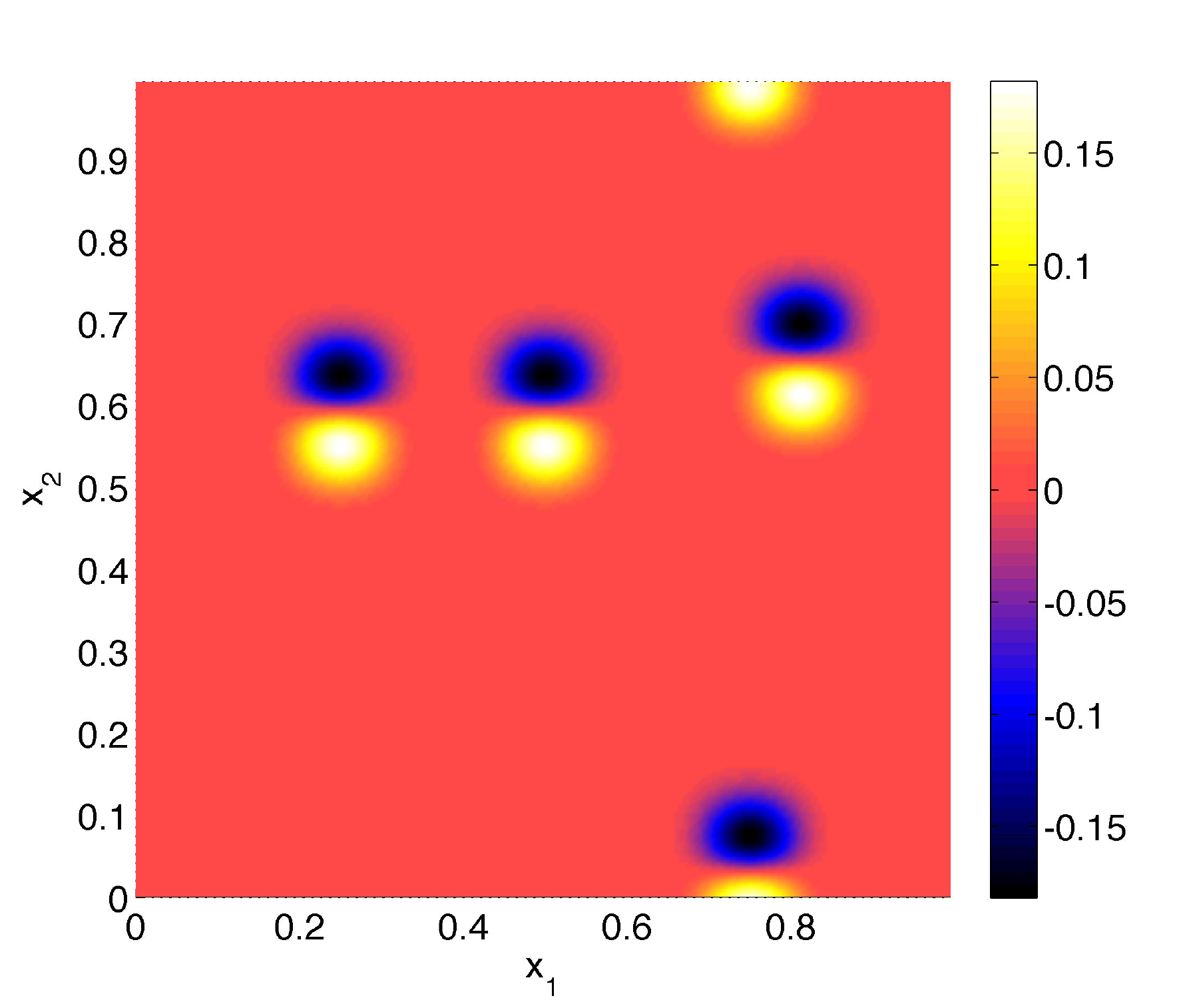}
}
\subfigure[The evolution of the interface through one period for $F=0.04$. Snapshots were taken at equal time intervals, so one can see that the interface spends most of its time near the critical pinned state (see Figure~\ref{fig:standard_stuck})]
{
\label{fig:standard_evolution}
\includegraphics[width=0.43\textwidth]{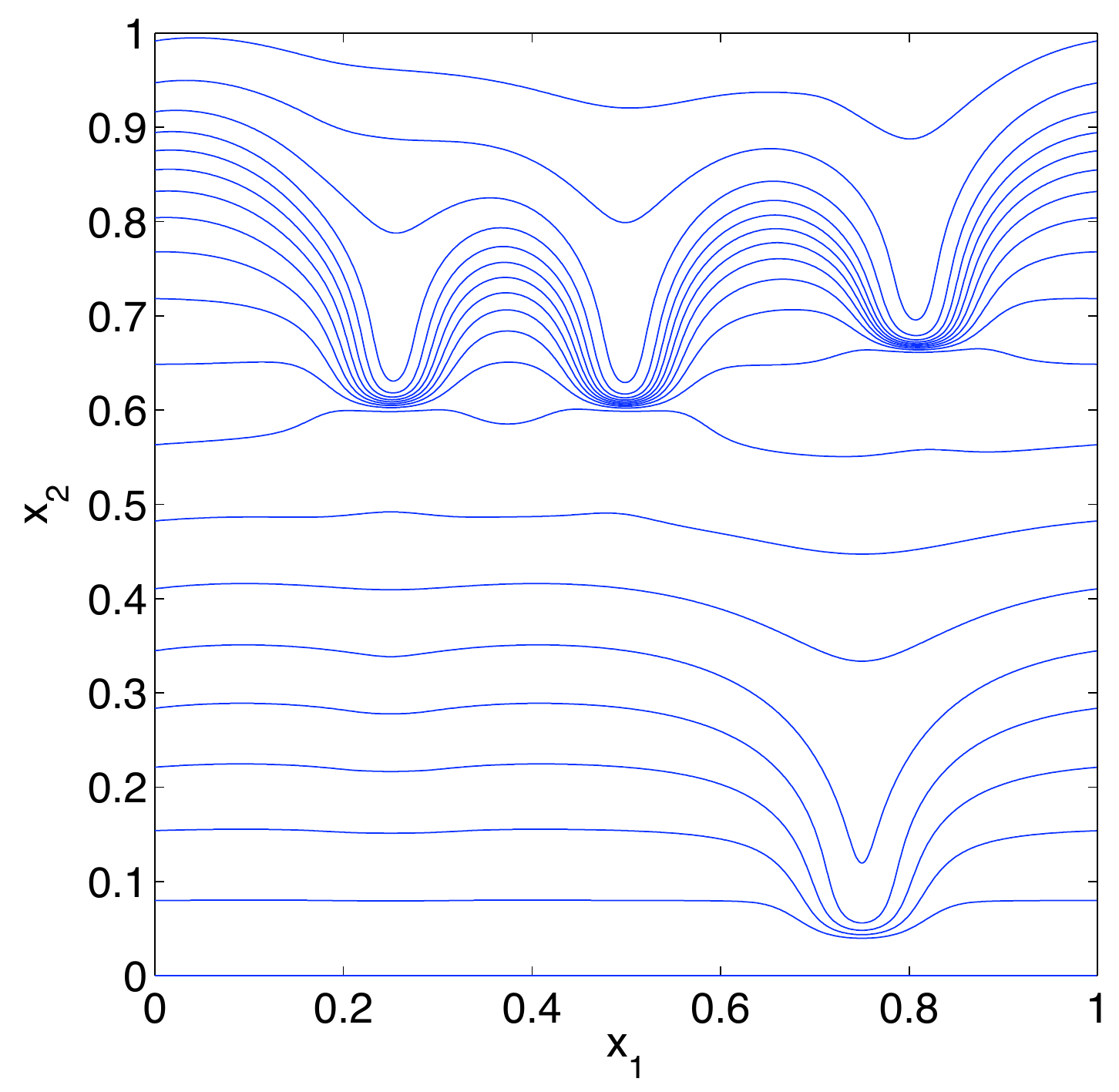}
}
\subfigure[Interface stuck at the inclusions for $F=0.03$]
{
\label{fig:standard_stuck}
\includegraphics[width=0.43\textwidth]{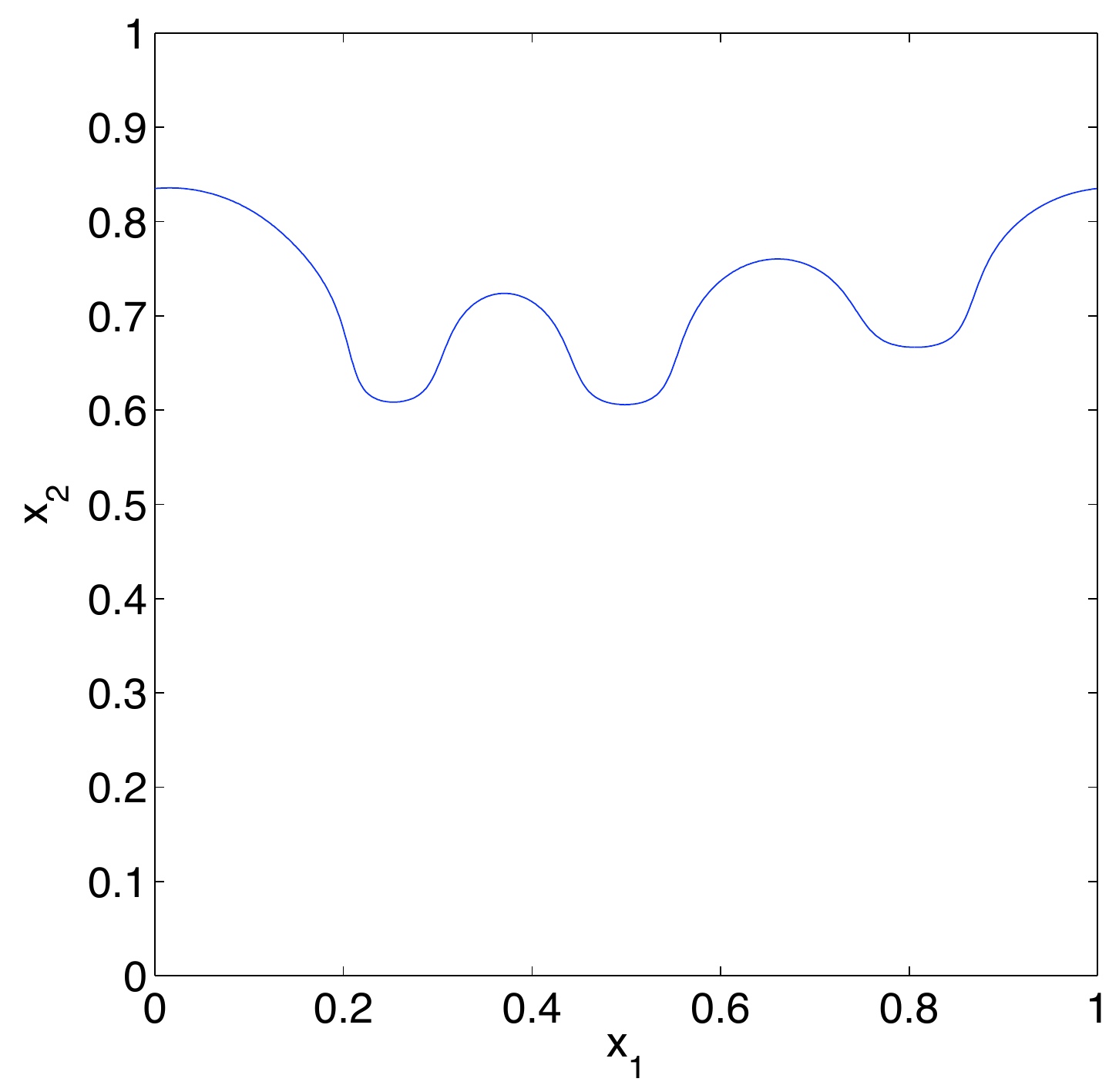}
}
\subfigure[Power law of the depinning transition]
{
\label{fig:standard_depinning_power}
\includegraphics[width=0.43\textwidth]{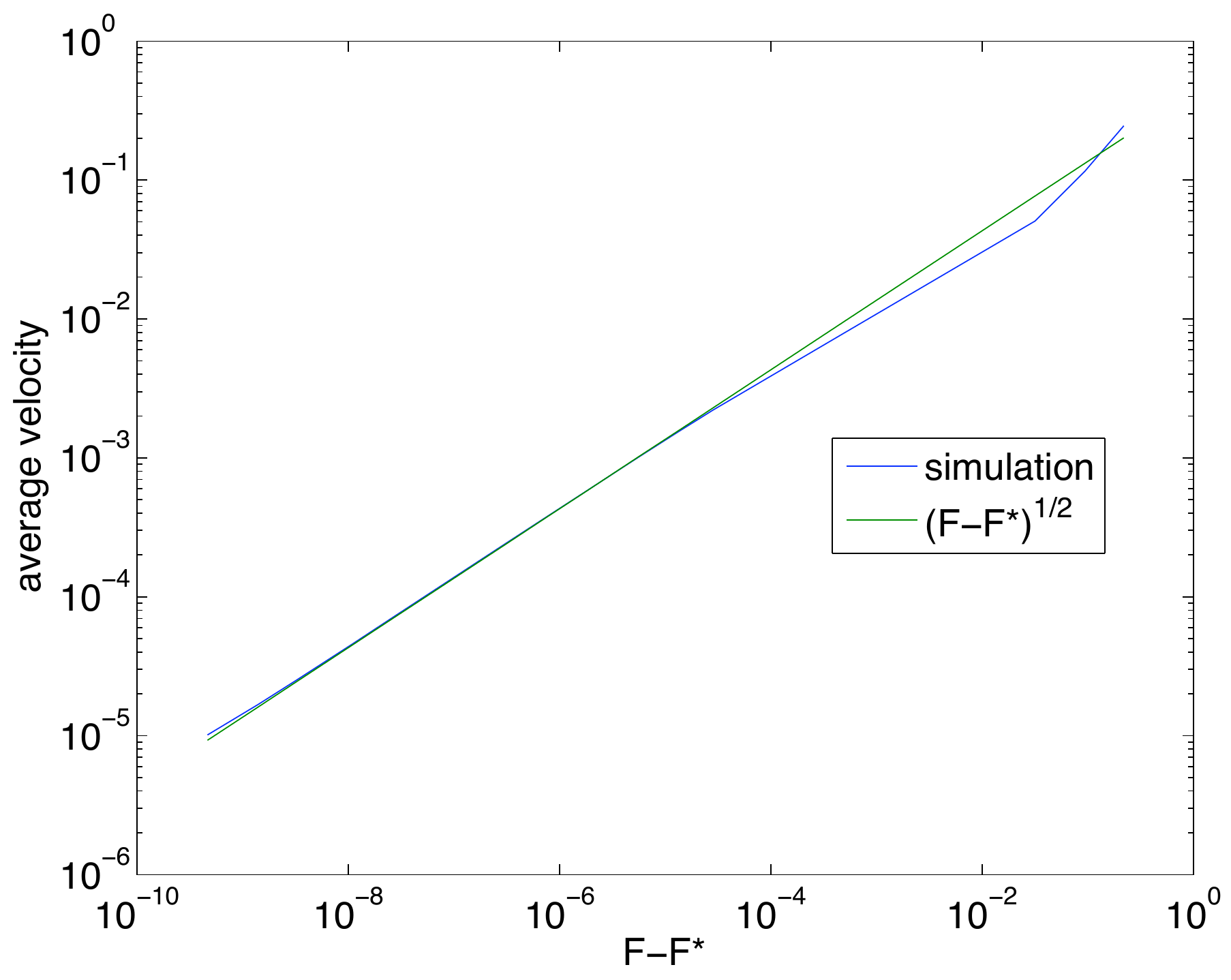}
}
\caption[Experiment 1, the general depinning behavior]{Experiment 1, the general depinning behavior}
\end{figure}
\begin{figure}
\centering
\subfigure[$1/8$]
{
\includegraphics[width=0.43\textwidth]{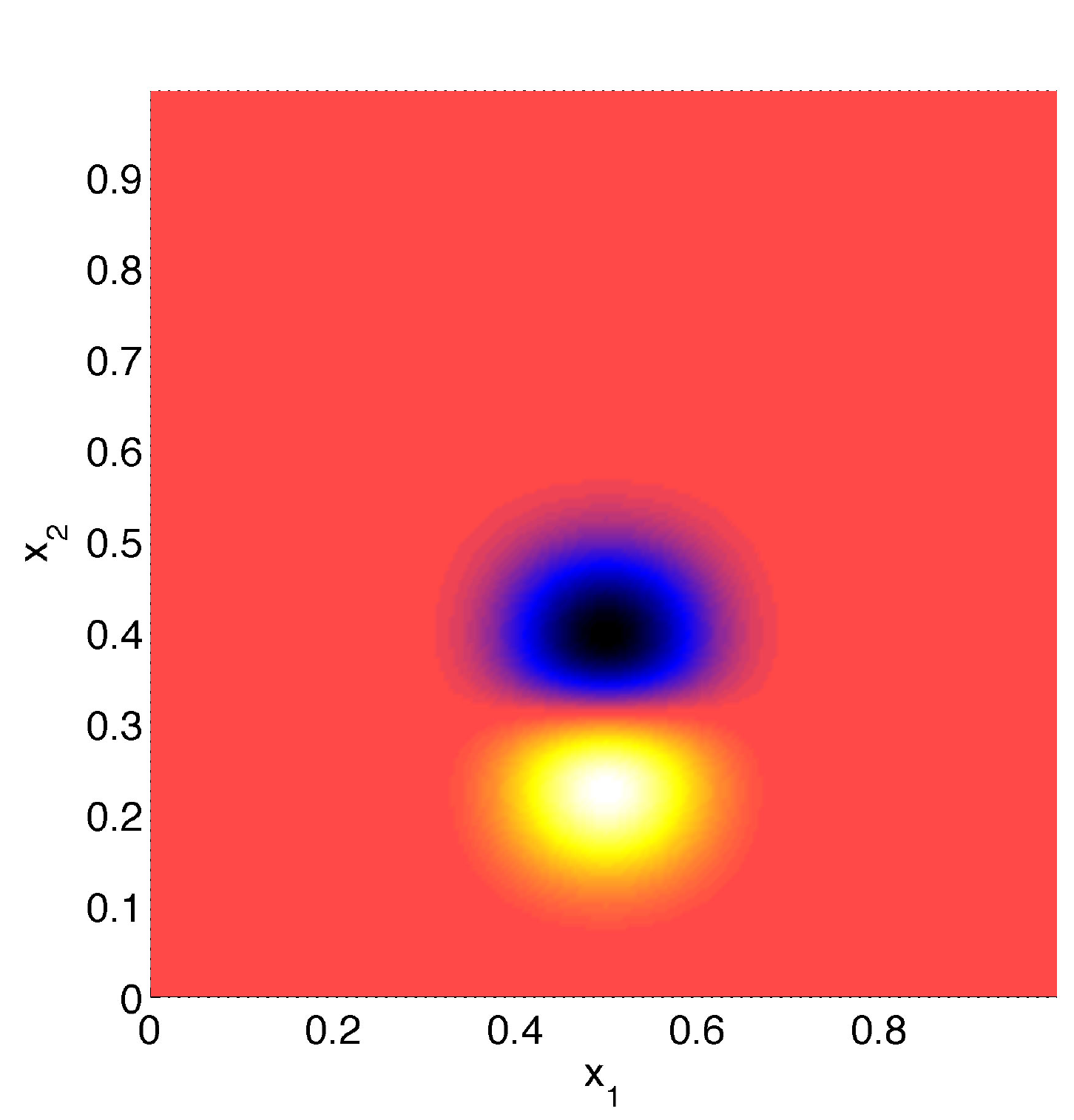}
}
\subfigure[$1/32$]
{
\includegraphics[width=0.43\textwidth]{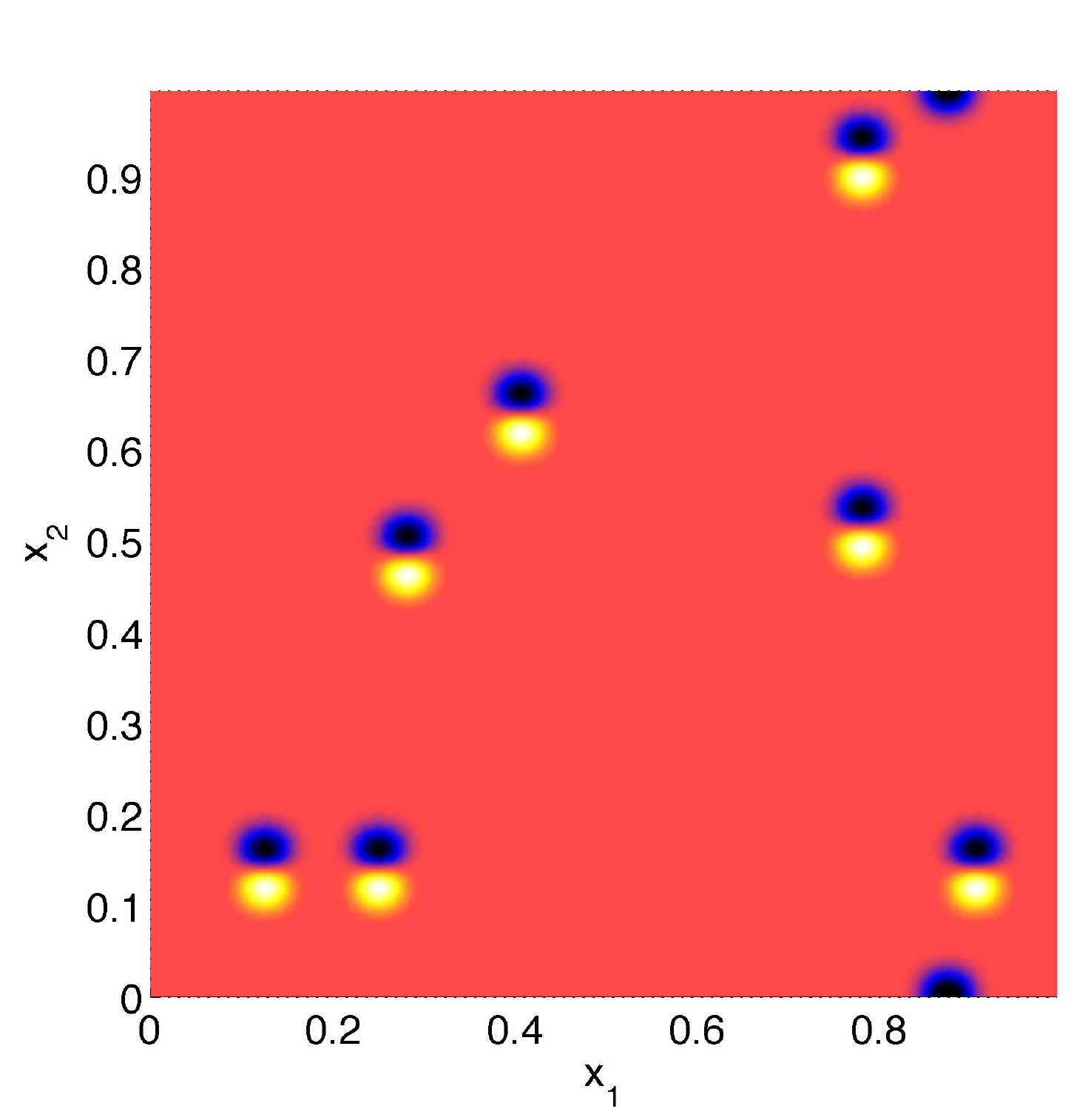}
}
\subfigure[$1/64$]
{
\includegraphics[width=0.43\textwidth]{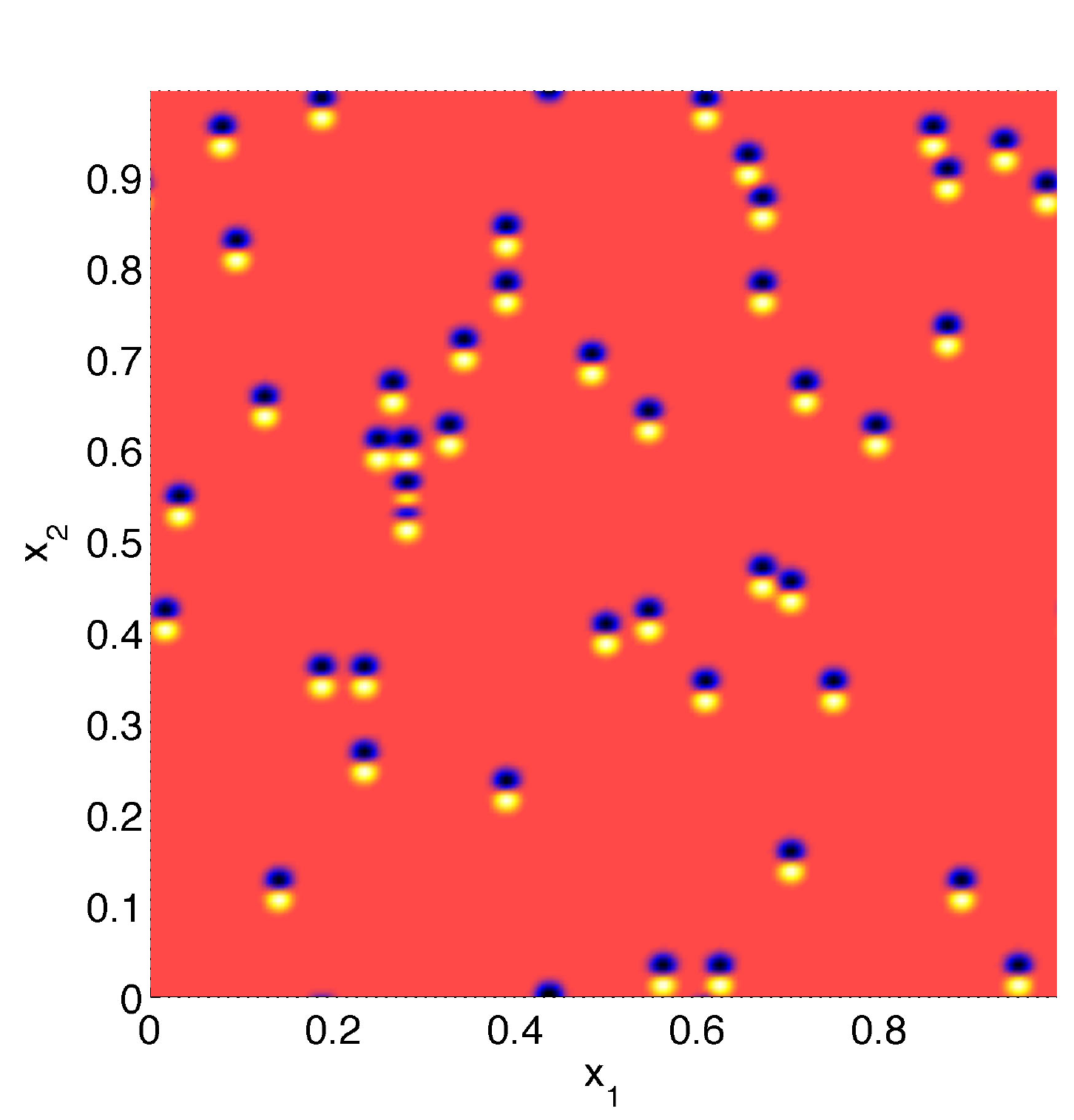}
}
\subfigure[$1/128$]
{
\includegraphics[width=0.43\textwidth]{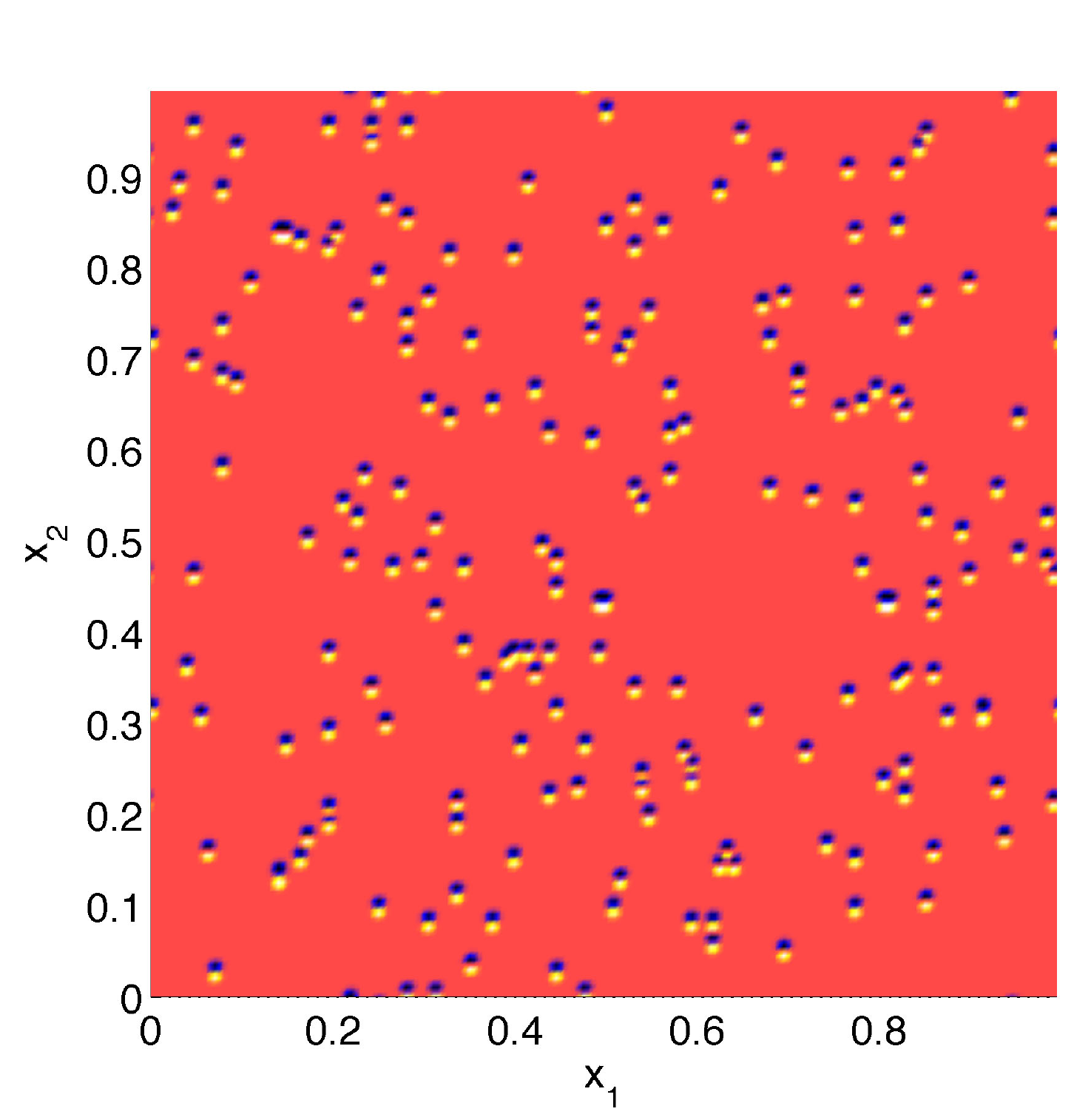}
}
\caption[The local pinning forces used in Experiment 2]{The local pinning forces used in Experiment 2. The maximum and minimum force are the same for all simulations, the grid on which the spline is discretized, however, is of the indicated size. Pinning sites are then distributed randomly with constant probability of occurrence. The pinning force for discretization size can be found in Figure~\ref{fig:standard_pinningforce}, where also a relation between the color and the magnitude of the force is given.}
\label{fig:size}
\end{figure}
\begin{table}
\centering
\begin{tabular}{cc}
Size of pinning sites & Critical applied force $F^*$ \\
$1/8$ &  $0.0204$ \\
$1/16$ &  $0.0307$ \\
$1/32$ &   $0.0150$\\
$1/64$ &   $0.0143$\\
$1/128$ &  $0.0119$\\
\end{tabular}
\caption[Dependence of the critical applied force on the size of the pinning sites]{Dependence of the critical applied force on the size of the pinning sites}
\label{tab:size_crit}
\end{table}
\begin{figure}
\centering
\includegraphics[width=0.6\textwidth]{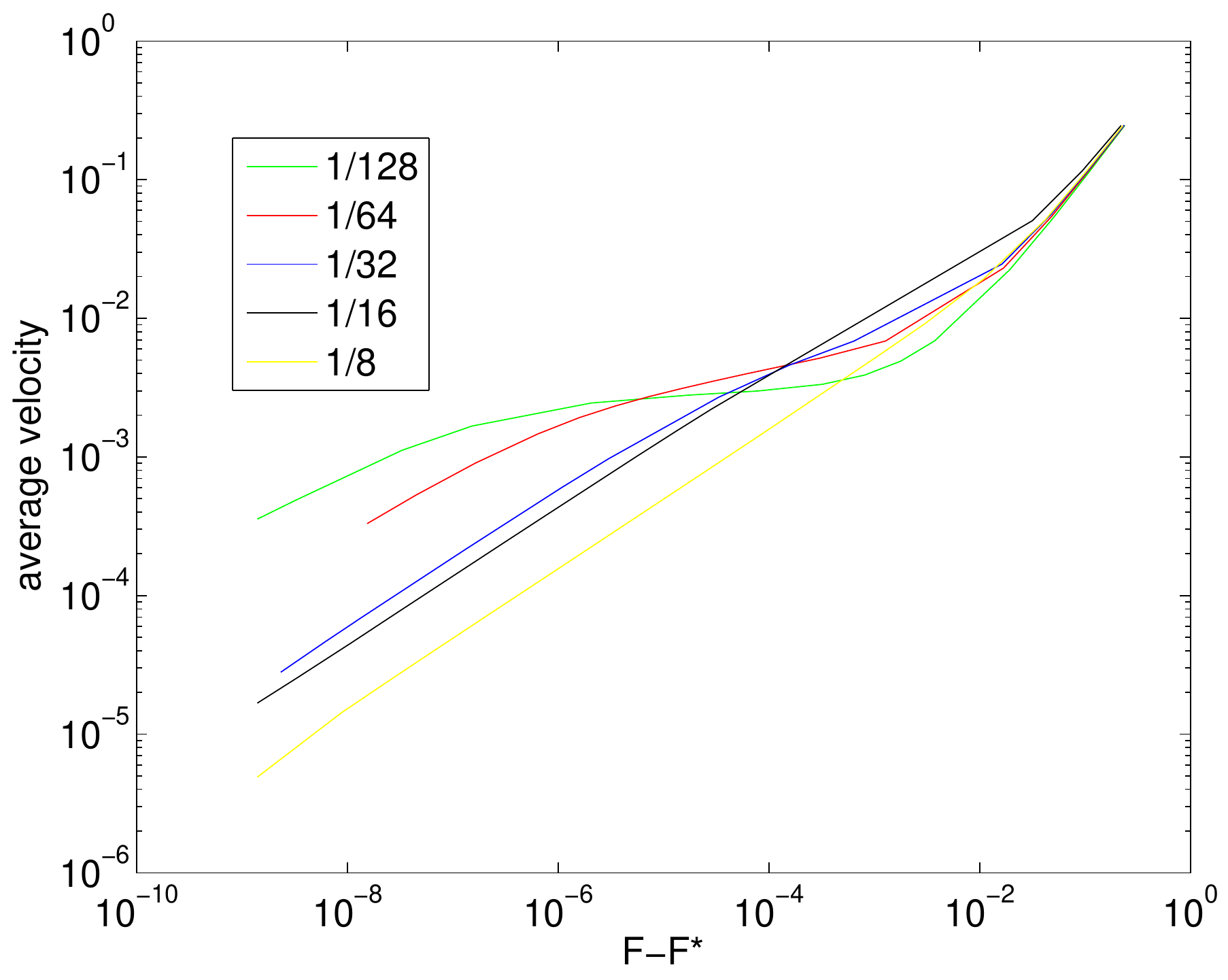}
\caption[The depinning behavior for different sizes of pinning sites]{The depinning behavior for different sizes of pinning sites. One can see that the onset square root power law behavior is pushed towards lower applied forces for smaller (and therefore sharper) pinning sites.}
\label{fig:size_depinning}
\end{figure}

\begin{description}
\item{\bf Experiment 1: General depinning behavior.}
As a standard example, we use a local driving force $\varphi(x_1, x_2) = \frac{\partial}{\partial x_2} \Phi(x_1, x_2)$, where $\Phi$ is a potential that has smooth dips of a fixed depth and radius at random points. We approximate $\varphi$ by a cubic $C^2$ spline curve. The exact force used in this simulation is depicted in Figure~\ref{fig:standard_pinningforce}. The constants used for this simulation are shown in Table~\ref{tab:general_depinning}. The evolution of the interface through one period is shown in Figure~\ref{fig:standard_evolution}, where one can see that the interface spends most of its time near the critical stuck state depicted in Figure~\ref{fig:standard_stuck}. In Figure~\ref{fig:standard_depinning_power}, the relation between the average velocity, compared to a square-root power law is shown. The fit over almost three decades is excellent and at the very high end of the applied force one can see that the velocity turns toward a linear dependence on the applied force, as expected.

\item{\bf Experiment 2: Comparison of the depinning behavior for different sizes of pinning sites}
For this experiment, we use five different sizes pinning sites, yielding the force distribution as depicted in Figure~\ref{fig:size}. All other parameters are kept as in Experiment 1. The dependence of the critical external force $F^*$ is shown in Table~\ref{tab:size_crit}. The depinning behavior is shown in Figure~\ref{fig:size_depinning}.
\end{description}


\section*{Acknowledgements}
This work draws from the doctoral thesis of Patrick Dondl at the California Institute of Technology.  It is a pleasure to acknowledge discussions with Bogdan Craciun, Nicolas Dirr and Aaron Yip.  We gratefully acknowledge the financial support of the National Science Foundation (ACI-0204932, DMS-0311788),

\bibliography{pinning}
\bibliographystyle{plain}

\end{document}